\documentclass{amsart}
\usepackage{cmd_pkg}

\title{A note on equivalences between various mixing scales}
\author{Bohan Zhou }
\email{Bohan.Zhou@dartmouth.edu}
\address{Department of Mathematics, Dartmouth College, Hanover, NH 03755}
\date{December 31, 2019}

\begin{document}

\maketitle
\section{Introduction}
To quantify the degree of homogenization of scalars and the efficiency of mixing phenomenon in many field of science and engineering, a natural choice is the variance of the concentration of scalars. Let $\rho$ denote the density function of a scalar on a domain $\Omega\subseteq \R^d$, then the variance is 
\[\mathrm{Var}(\rho)=\frac{1}{\abs{\Omega}}\int_{\Omega}\rho^2(x)\diff{x}-\left(\frac{1}{\abs{\Omega}}\int_{\Omega}\rho(x)\diff{x}\right)^2,\]
where $\abs{\Omega}$ denotes the Lebesgue measure of $\Omega$. A heuristic observation is that the smaller the variance is, the better the homogenization or mixing is.

This measurement works well when the mechanism has diffusion effect. However, when it applies on the transport equation, it fails due to the conservation of any $L^p$ norm of scalars. People still hope to measure the mixing even without any diffusion, since the density $\rho$ could converge to zero weakly only with transport effect. Hence we need a measurement to detect small scales. Note that mixing here needs not require miscible fluids. 

Mathew et al. in \cite{mathew2005multiscale} first introduced a homogeneous negative Sobolev norm $\norm{\cdot}_{\dot{H}^{-1/2}}$ to quantify the level of mixedness. After that, Lin et al. of \cite{lin2011optimal} generalized it to $\norm{\cdot}_{\dot{H}^{-\alpha}}$ for $\alpha\in (0,1]$. And there is a detailed review paper \cite{thiffeault2012review} by Thiffeault addressing about related multiscale norms and its applications in many interesting problems.

In the meanwhile, Bressan in his conjecture \cite{bressan2003lemma} introduced a more geometric sense measurement for density function with value $\pm 1$, which is related with density of a set in geometric measure sense. Yao et al. in \cite{yao2014mixing} generalized this geometric measurement for bounded density functions.

Note though both measurements can be regarded as a scale of length, to quantify small scales. However, Lunasin et al. in \cite{lunasin2012optimal} provided with two simple but impressive examples, to demonstrate that there is no simple equivalence between those two measurements, that is,
\[C_1 \norm{\cdot}_1\leqslant \norm{\cdot}_2 \leqslant C_2 \norm{\cdot}_1\]
does not hold for some constant $C_1,C_2>0$. This is the starting point of this expository paper.

\section{Preliminaries}
In the paper, we work on $d$-dimensional flat torus $\T^d=\R^d/\Z^d$ or periodic box $[0,\lambda]^d$ of side-length $\lambda$.

For $s\in\R$, we say that a distribution $\rho$ belongs to $\dot{H}^s(\T^d)$ if
\begin{align*}
    \norm{\rho}_{\dH^s(\T^d)}\defeq\sum_{k\in \Z^d}\abs{k}^{2s}\abs{\hat{\rho}(k)}^2<\infty.
\end{align*}
Any function $f\in L^2(\T^d)$ with $\int_{\T^d} f\diff{x}=0$ belongs to $\dH^s(\T^d)$ for $s<0$.

For any $r>0$, we define the averaging operators $A_r$ for any scalar field $\rho\in L^{\infty}(\T^d)$:
\[A_r \rho(x):=\avintdis_{B(x,r)}\rho(y)\diff{y}=\frac{\int_{B(x,r)}\rho(y)\diff{y}}{\abs{B(x,r)}}.\]

\begin{definition}
For any scalar field $\rho$ on $\T^d$ with zero mean $\int_{\T^2}\rho\diff{x}=0$, we call $\norm{\rho}_{\dH^{-1}}$ as the functional mixing scale.
\end{definition}

\begin{definition}
For any scalar field $\rho\in L^{\infty}(\T^d)$, given an accuracy parameter $0<\kappa<1$, the \textit{geometric mixing scale} of $\rho$, denoted by $\G(\rho;\kappa)$, is the infimum of all $r>0$ such that for every $x\in\T^d$, there holds
\begin{align}
\label{GMS}
    \abs{A_r\rho(x)}\leqslant\kappa \norm{\rho}_{L^{\infty}}.
\end{align}
And we say the scalar field $\rho$ is $\kappa$-mixed to the scale $\G(\rho;\kappa)$.
\end{definition}

In both scales, the smaller the mixing scale of the scalar field is, the better the scalar field is expected to be mixed.

However, the following example shows a defect of the geometric mixing scale in the above definition. Roughly speaking, for any $r>\G(\rho;\kappa)$, a scalar field that is well mixed in a small scale is not necessary well mixed in a larger scale.
\begin{example}
There exist a scalar field $\rho\in L^{\infty}(\T^d)$ and $0<\kappa<1$ such that there exists an $r>\G(\rho;\kappa)$ with
\begin{align*}
    \abs{A_r\rho(x)}>\kappa \norm{\rho}_{L^{\infty}}\qquad\mathrm{for~some~}x\in\T^d.
\end{align*}
Moreover, let $A$ be the set of such $x$. Then we have $\abs{A}>0$.
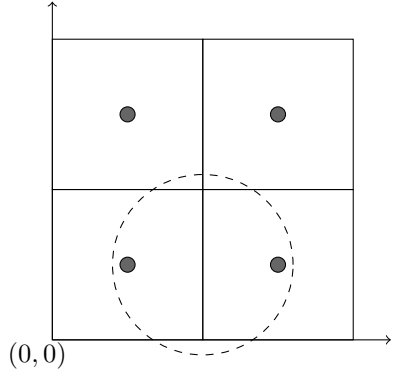
\begin{figure}[h]
\captionsetup{justification=centering}
\centering
\begin{tikzpicture}

\draw  (0,0) rectangle (2,2);
\draw (2,0) rectangle (4,2);
\draw (0,2) rectangle (2,4);
\draw (2,2) rectangle (4,4);

\draw[fill=black!60!white] (1,1) circle (0.1);
\draw[fill=black!60!white] (3,1) circle (0.1);
\draw[fill=black!60!white] (1,3) circle (0.1);
\draw[fill=black!60!white] (3,3) circle (0.1);
\draw[dashed] (2,1) circle (1.2);

\draw[->] (0,0) -- (4.5,0);
\draw[->] (0,0) -- (0,4.5);
\node at (-0.2,-0.2) {$(0,0)$};

\end{tikzpicture}
\caption{\label{fig: defect}This example shows a defect of the geometric mixing scale. Each box is of length 1, each solid ball is of radius $\varepsilon$, the dashed circle is of radius $\frac{1}{2}+\varepsilon$. The scalar field is well-mixed at $\G(\rho;\kappa)$, which is less than $\frac{1}{2}$, while it is not well-mixed at $\frac{1}{2}+\varepsilon$.}
\end{figure}
\end{example}
\begin{proof}
Let $\kappa=\frac{0.05^2\pi}{1-0.05^2\pi}\approx 0.0079$ and periodic scalar field $\rho(x)$ on $[0,1]^2$ is given by
\begin{equation}
    \rho(x)=\left\{
    \begin{aligned}
        &1,\qquad\qquad\quad x\in B_{\varepsilon}(\frac{1}{2},\frac{1}{2});\\
        &\frac{-\pi\varepsilon^2}{1-\pi\varepsilon^2},\qquad x\in [0,1]^2\setminus B_{\varepsilon}(\frac{1}{2},\frac{1}{2}).
    \end{aligned}
    \right.
\end{equation}
where $\varepsilon=\sqrt{\frac{\kappa}{(1+\kappa)\pi}}=0.05$. Note that $\norm{\rho}_{L^{\infty}}=1$.

First, we check that $\rho(x)$ is mean-zero on $[0,1]^2$ by
\[\int_{[0,1]^2}\rho(x)\diff{x}=\pi\varepsilon^2-\frac{\pi\varepsilon^2}{1-\pi\varepsilon^2}(1-\pi\varepsilon^2)=0.\]

Second, we claim $\G(\rho;\kappa)\leqslant \frac{1}{2}$.

We call each ball with density 1 by a \textit{solid ball}. As shown in the \prettyref{fig: defect}, given any $x\in [0,1]^2$, the ball $B_{0.5}(x)$ cannot contain more than one entire solid ball by the geometric relationship.

When $B_{0.5}(x)$ does not contain any portion of any solid ball, we have
\[\abs{A_{0.5}\rho(x)}=\frac{\pi\varepsilon^2}{1-\pi\varepsilon^2}=\kappa=\kappa\norm{\rho}_{L^{\infty}}.\]

When $B_{0.5}(x)$ contains an entire solid ball, note that $\rho(x)$ is mean-zero on $[0,1]^2$, we have
\[\abs{A_{0.5}\rho(x)}=\frac{(1-0.5^2\pi )\frac{\pi\varepsilon^2}{1-\pi\varepsilon^2}}{0.5^2\pi }<\kappa=\kappa\norm{\rho}_{L^{\infty}}.\]

By interpolation, for any $x\in [0,1]^2$, we have $\abs{A_{0.5}\rho(x)}\leqslant \kappa\norm{\rho}_{L^{\infty}}$, which implies that $\G(\rho;\kappa)\leqslant \frac{1}{2}$ by definition.

For $r=\frac{1}{2}+\varepsilon>\G(\rho;\kappa)$ and we pick ball $B_r(1,\frac{1}{2})$, then
\[\abs{A_{r}\rho(1,\frac{1}{2})}=\frac{1}{\pi r^2}\abs{2\pi\varepsilon^2-\frac{\pi\varepsilon^2}{1-\pi\varepsilon^2}(\pi r^2-2\pi\varepsilon^2)}\approx 0.0087>\kappa=\kappa\norm{\rho}_{L^{\infty}}.\]

Furthermore, by the construction, it is easy to see the set consisting of such $x$ is not a negligible set.
\end{proof}

Based on the above observation, we introduce another geometric mixing scale, the so-called \textit{strong geometric mixing scale}, denoted by $\SG(\rho)$.

\begin{definition}
For any given $0<\kappa<1$, the strong geometric mixing scale $\mathcal{SG}(\rho;\kappa)$ is given by:
\begin{align}
    \label{SGMS}
    \SG(\rho;\kappa)\defeq \sup_{x\in\T^d}\inf\left\{\varepsilon_x \bigg\rvert  \forall r\geqslant \varepsilon_x, \abs{A_r\rho(x)}\leqslant\kappa\norm{\rho}_{L^{\infty}}\right\}.
\end{align}
\end{definition}

\begin{lemma}
\label{lem: UBSGMS}
Let $\rho\in L^{\infty}([0,\lambda]^d)$ be a mean-zero scalar field. Given $0<\kappa<1$, for $r\geqslant\frac{\lambda\sqrt{d}}{1-(1-\kappa)^{1/d}}$ and for any $x\in [0,\lambda]^d$, we have \[\abs{A_r\rho(x)}\leqslant\kappa\norm{\rho}_{L^{\infty}}.\]
\end{lemma}
\begin{proof}
For $r\geqslant \frac{\lambda\sqrt{d}}{1-(1-\kappa)^{1/d}}>\lambda\sqrt{d}$, given any $d$-ball with radius $r$, we pack it with the $d$-cubes of side length $\lambda$. Let $\mathcal{A}$ be the union of cubes which are contained in the $d$-ball with radius $r$ entirely. Let $\mathcal{B}$ be the union of cubes, parts of which are outside of $d$-ball with radius $r$. Since the maximum distance in any $d$-cube is $\lambda\sqrt{d}$, the distance between the sphere of $d$-ball with radius $r$ and $\mathcal{A}$ is less than $\lambda\sqrt{d}$. Please refer to \prettyref{fig: pack}.

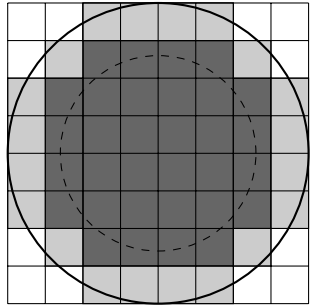
\begin{figure}[h]
\centering
\begin{tikzpicture}[scale=2]

\draw[fill=black!60!white] (0.5,0.25) rectangle (1.5,1.75);
\draw[fill=black!60!white] (0.25,0.5) rectangle (0.5,1.5);
\draw[fill=black!60!white] (1.5,0.5) rectangle (1.75,1.5);

\draw[fill=black!20!white] (0,0.5) rectangle (0.25,1.5);
\draw[fill=black!20!white] (1.75,0.5) rectangle (2,1.5);
\draw[fill=black!20!white] (0.5,0) rectangle (1.5,0.25);
\draw[fill=black!20!white] (0.5,1.75) rectangle (1.5,2);
\draw[fill=black!20!white] (0.25,1.5) rectangle (0.5,1.75);
\draw[fill=black!20!white] (0.25,0.25) rectangle (0.5,0.5);
\draw[fill=black!20!white] (0.5,1.75) rectangle (1.5,2);
\draw[fill=black!20!white] (1.5,0.25) rectangle (1.75,0.5);
\draw[fill=black!20!white] (1.5,1.5) rectangle (1.75,1.75);

\draw[thick] (1,1) circle (1);
\draw[dashed] (1,1) circle (0.65);
\draw [step=0.25] (0,0) grid (2,2);

\end{tikzpicture}
\caption{$d$-cubes with side length $\lambda$ are packed into a $d$-ball with radius $r$. $\mathcal{A}$ is in dark grey; $\mathcal{B}$ is in light grey. The radius of solid circle is $r$; the radius of dashed circle is $r-\lambda\sqrt{d}$.\label{fig: pack}}
\end{figure}

Let $\alpha_d$ denote the Lebesgue measure of unit $d$-ball. We note that $\rho$ is mean-zero in each $d-$cube with side length $\lambda$, thus
\begin{align*}
    \abs{A_r(\rho)(x)}&=\frac{\abs{\int_{\mathcal{A}}\rho(y)\diff{y}+\int_{\mathcal{B}\cap B(x,r)}\rho(y)\diff{y}}}{\alpha_d r^d}\\
    &=\frac{\abs{\int_{\mathcal{B}\cap B(x,r)}\rho(y)\diff{y}}}{\alpha_d r^d}\\
    &\leqslant \frac{\abs{\mathcal{B}\cap B(x,r)}}{\alpha_d r^d}\norm{\rho}_{L^{\infty}}\\
    &\leqslant \frac{\alpha_d(r^d-(r-\lambda\sqrt{d})^d)}{\alpha_d r^d}\norm{\rho}_{L^{\infty}}\\
    &=(1-(1-\frac{\lambda\sqrt{d}}{r})^d)\norm{\rho}_{L^{\infty}}\\
    &\leqslant \kappa\norm{\rho}_{L^{\infty}}.
\end{align*}
\end{proof}
\begin{remark}
In particular, for $d=2$, $r\geqslant \frac{\sqrt{2}\lambda}{1-(1-\kappa)^{1/2}}=\frac{\sqrt{2}\lambda(1+(1-\kappa)^{1/2})}{\kappa}$. This estimation is slightly better than the estimation in \cite{Alberti2019Expoential} where $r\geqslant \frac{4\sqrt{2}\lambda}{\kappa}$.
\end{remark}
\begin{corollary}
Let $\rho\in L^{\infty}([0,\lambda]^d)$ be a mean-zero scalar field. Given $0<\kappa<1$, $\SG(\rho;\kappa)$ exists and $\SG(\rho;\kappa)\leqslant \frac{\lambda\sqrt{d}}{1-(1-\kappa)^{1/d}}$. Moreover, for any $r\geqslant \SG(\rho;\kappa)$ and any $x\in [0,\lambda]^d$,
\begin{align*}
    \abs{A_r\rho(x)}\leqslant\kappa \norm{\rho}_{L^{\infty}}.
\end{align*}
\end{corollary}

\section{Main theorem}

\begin{theorem}[Equivalence between various scales]
For a 2D mean-zero scalar field $\rho(t,\cdot)$ bounded uniform-in-time in $L^{\infty}(\mathbb{T}^2)$, assume that $\rho(t,\cdot)$ is the solution to the Cauchy problem of the transport equation with initial data $\rho_0(x)$, then the following statements are equivalent:
\begin{enumerate}
\item $\rho(t,\cdot)\rightharpoonup 0$ as $t\to\infty$ in $L^2(\mathbb{T}^2)$, i.e., $\displaystyle\lim_{t\to\infty}\int_{\T^2}\rho(x,t)\phi(x)\diff{x}=0,\quad\forall \phi\in L^2(\T^2)$.
\item $\lim_{t\to\infty}\norm{\rho(t,\cdot)}_{\dot{H}^{-1}(\mathbb{T}^2)}=0.$
\item For all $\kappa\in (0, 1)$, $\displaystyle\lim_{t\to\infty}\mathcal{G}(\rho(t,\cdot))= 0$.
\item For all $\kappa\in (0, 1)$, $\displaystyle\lim_{t\to\infty} \mathcal{SG}(\rho(t,\cdot))= 0.$
\end{enumerate}
\end{theorem}
\begin{proof}
``$1\Longleftrightarrow 2$'' is proved in \cite{lin2011optimal}.

``$1\Longrightarrow 3$'': For any $\kappa\in (0,1)$ and fixed $\varepsilon>0$, there exists $\delta=\kappa$ such that: for any $x\in \T^2$, there exists a $T_0$ such that for any $t\geqslant T_0$, we have $\abs{\int_{\T^2}\rho(y,t)\phi(y)\diff{y}}\leqslant \delta$ for any $\phi(y)\in L^2(\T^2)$. We pick

\[\phi(y;x)=\frac{\Id_{B(x,\varepsilon)}(y)}{\abs{B(x,\varepsilon)}\norm{\rho}_{L^{\infty}}}\in L^2(\T^2).\]
Thus
\[\abs{\int_{\T^2}\rho(y;t)\phi(y)\diff{y}}=\abs{\frac{\int_{B(x,\varepsilon)}\rho(y,t)\diff{y}}{\abs{B(x,\varepsilon)}\norm{\rho}_{\infty}}}\leqslant \delta,\]
which implies that $\abs{A_{\varepsilon}\rho(x;t)}\leqslant \delta\norm{\rho}_{\infty}=\kappa\norm{\rho}_{\infty}$ for any $x\in\T^2$. Furthermore, by the definition of geometric mixing scale, $\G(\rho(t);\kappa)\leqslant \varepsilon$. 

Since we can pick $\varepsilon$ sufficiently small, \[\displaystyle\lim_{t\to\infty}\G(\rho(t))= 0.\]

``$1\Longrightarrow 4$'': This proof is the same with ``$1\Longrightarrow 3$''.

``$4\Longrightarrow 1$'': This proof is inspired by \cite{Crippa2019Polynomial}.

Since continuous function is dense in $L^2(\T^2)$, it suffices to only consider continuous test function $\phi\in L^2(\T^2)$. Given $\delta>0$, we need to show there exists a $T_0>0$ such that for $t\geqslant T_0$, we have
\[\abs{\int_{\T^2} \rho(x,t)\phi(x)\diff{x}}\leqslant \delta.\]

Since $\phi$ is uniformly continuous, there exists a $\varepsilon$ such that whenever $y\in B(x,\varepsilon)$ implies $\abs{\phi(x)-\phi(y)}\leqslant \frac{\delta}{3\norm{\rho}_{\infty}}$. Furthermore, we choose a finite family of disjoint ball $\{B(x_i,r_i)\}_{i=1}^n$ such that
\[B(x_i,r_i)\subset \T^2,\qquad r_i\leqslant\varepsilon\qquad\mathrm{and}\qquad \abs{\T^2\backslash\bigcup_{i=1}^n B(x_i,r_i)}\leqslant \frac{\delta}{3\norm{\rho}_{\infty}\norm{\phi}_{\infty}}.\]

So the integral
\[\abs{\int_{\T^2}\rho(t,x)\phi(x)\diff{x}} \leqslant \underbrace{\abs{\int_{\cup_{i=1}^n B(x_i,r_i)}\rho(t,x)\phi(x)\diff{x}}}_{\RN{1}} + \underbrace{ \abs{\int_{\T^2\backslash\bigcup_{i=1}^n B(x_i,r_i)}\rho(t,x)\phi(x)\diff{x}}}_{\RN{2}},\]
For the second part, we have
\[\RN{2}\leqslant\norm{\rho}_{\infty}\norm{\phi}_{\infty}\abs{\T^2\backslash \bigcup_{i=1}^n B(x_i,r_i)}\leqslant\frac{\delta}{3};\]
For the first part, we have
\[
\begin{aligned}
\RN{1}\leqslant& \sum_{i=1}^n\abs{\int_{B(x_i,r_i)}(\phi(y)-\phi(x_i)+\phi(x_i))\rho(t,y)\diff{y}}\\
\leqslant & \underbrace{\sum_{i=1}^n \abs{\phi(x_i)}\abs{\int_{B(x_i,r_i)}\rho(t,y)\diff{y}}}_{\RN{1}_a} +\underbrace{\sum_{i=1}^n \max_{y\in B(x_i,r_i)}{\abs{\phi(y)-\phi(x_i)}}\int_{B(x_i,r_i)}\abs{\rho(t,y)}\diff{y}}_{\RN{1}_b}\\
\end{aligned}
\]
Since $\displaystyle\lim_{t\to\infty} \SG(\rho(t,\cdot))= 0$ for all $\kappa\in (0,1)$, there exist a $T_0=T_0(\kappa)$ such that for $t\geqslant T_0$, $\SG(\rho(t,\cdot))\leqslant \min r_i$. By the property of strong geometric mixing scale,

\[\RN{1}_a=\sum_{i=1}^n \abs{\phi(x_i)}\abs{\int_{B(x_i,r_i)}\rho(t,y)\diff{y}}\leqslant \kappa\norm{\rho}_{\infty}\sum_{i=1}^n\abs{\phi(x_i)}\abs{B(x_i,r_i)}\leqslant\frac{\delta}{3},\]
where it is trivial if $\phi(x_i)=0$ for all $i$, otherwise we pick $\kappa=\frac{\delta}{3\norm{\rho}_{\infty}\sum_{i=1}^n \abs{\phi(x_i)}\abs{B(x_i,r_i)}}$.

\[\RN{1}_b=\sum_{i=1}^n \max_{y\in B(x_i,r_i)}{\abs{\phi(y)-\phi(x_i)}}\int_{B(x_i,r_i)}\abs{\rho(t,y)}\diff{y}\leqslant\frac{\delta}{3\norm{\rho}_{\infty}}\sum_{i=1}^n\int_{B(x_i,r_i)}\abs{\rho(t,y)}\diff{y}\leq\frac{\delta}{3}.\]

Combine $\RN{1}_a, \RN{1}_b$ and $\RN{2}$, we conclude that for any given $\delta>0$, there exists a $T_0>0$ such that for all $t>T_0$ we have
\[\abs{\int_{\T^2}\rho(x,t)\phi(x)\diff{x}}\leqslant \delta.\]

``$3\Longrightarrow 1$'': We cannot repeat the proof of ``$4\Longrightarrow 1$'' for geometric mixing scale directly, because the defect of geometric mixing scale causes that $G(\rho(t,\cdot))\leqslant \min r_i$ cannot imply that
\begin{align}
\label{eq: avg_cri}
    \int_{B(x_i,r_i)}\rho(t,y)\diff{y}\leqslant \kappa\norm{\rho}_{\infty}\abs{B(x_i,r_i)}.
\end{align}

However it is resolvable through the following lemma \cite{Crippa2017personal} by Gianluca Crippa in a private communication. Once we have shown that for each $i$, \eqref{eq: avg_cri} holds. The proof of ``$4\Longrightarrow 1$'' can be repeated.
\end{proof}

\begin{lemma}[\cite{Crippa2017personal}]
If for all $\kappa>0$, $\displaystyle\lim_{t\to\infty}\G(\rho(t,\cdot);\kappa)\to 0$, then $\forall r>0$, $\kappa>0$ there exist a $T_0$ such that
\[\abs{\int_{B(x,r)}\rho(t,y)\diff{y}}\leqslant \kappa\norm{\rho}_{
\infty}\abs{B(x,r)},\]
for all $x\in \T^2$ and $t\geqslant T_0$.
\end{lemma}
\begin{proof}
Fix $\kappa>0$ and $r>0$, since for $\kappa'=\frac{\kappa}{2}$ we have that $\G(\rho(t,\cdot),\kappa')\to 0$, there exists a time $T_0$ such that for all $t\geqslant T_0$, there exists a radius $\delta\leqslant \frac{\kappa}{12}r$ such that for all $x\in\T^2$
\[\abs{\int_{B(x,\delta)}\rho(t,y)\diff{y}}\leqslant \kappa'\norm{\rho}_{\infty}\abs{B(x,\delta)}=\frac{\kappa}{2}\norm{\rho}_{\infty}\abs{B(x,\delta)}.\]

For an arbitrary $x$ and $t\geqslant T_0$, we estimate
\[
\begin{aligned}
    \int_{B(x,r)}\int_{B(z,\delta)}\rho(t,y)\diff{y}\diff{z}&=\int_{B(x,r)}\int_{B(x,r+\delta)}\Id_{B(z,\delta)}(y)\rho(t,y)\diff{y}\diff{z}\\
    &=\int_{B(x,r+\delta)}\rho(t,y)\int_{B(x,r)}\Id_{B(y,\delta)}(z)\diff{z}\diff{y}\\
    &=\int_{B(x,r-\delta)}\rho(t,y)\int_{B(x,r)}\Id_{B(y,\delta)}(z)\diff{z}\diff{y}\\
    &\qquad\qquad\qquad+\int_{B(x,r+\delta)\backslash B(x,r-\delta)}\rho(t,y)\int_{B(x,r)}\Id_{B(y,\delta)}(z)\diff{z}\diff{y}\\
    &=\abs{B(z,\delta)}\int_{B(x,r-\delta)}\rho(t,y)\diff{y}\\
    &\qquad\qquad\qquad+\int_{B(x,r+\delta)\backslash B(x,r-\delta)}\rho(t,y)\int_{B(x,r)}\Id_{B(y,\delta)}(z)\diff{z}\diff{y}.
\end{aligned}
\]
Thus, we have
\[
\begin{aligned}
    \abs{\int_{B(x,r)}\rho(t,y)\diff{y}}&\leqslant \abs{\int_{B(x,r-\delta)}\rho(t,y)\diff{y}}+\norm{\rho}_{\infty}\abs{B(x,r)\backslash B(x,r-\delta)}\\
    &\leqslant \abs{\int_{B(x,r)}\avintdis_{B(z,\delta)}\rho(t,y)\diff{y}\diff{z}}+\norm{\rho}_{\infty}\left(\abs{B(x,r+\delta)\backslash B(x,r-\delta)}+\abs{B(x,r)\backslash B(x,r-\delta)}\right)\\
    &\leqslant \left(\frac{\kappa}{2}+4\frac{\delta}{r}+2\frac{\delta}{r}-(\frac{\delta}{r})^2\right)\norm{\rho}_{\infty}\abs{B(x,r)},
\end{aligned}
\]
Using that $\delta\leqslant \frac{\kappa}{12}r$, we conclude that 
\[\int_{B(x,r)}\rho(t,y)\diff{y}\leqslant \kappa\norm{\rho}_{\infty}\abs{B(x,r)}\]
for all $x\in\T^2$ and all $t\geqslant T_0$.
\end{proof}

\bibliographystyle{alpha}
\bibliography{cite}

\newcommand{\etalchar}[1]{$^{#1}$}
\begin{thebibliography}{LLN{\etalchar{+}}12}

\bibitem[ACM19]{Alberti2019Expoential}
G.~Alberti, G.~Crippa, and A.~L. Mazzucato.
\newblock Exponential self-similar mixing by incompressible flows.
\newblock {\em J. Amer. Math. Soc.}, 32(2):445--490, 2019.

\bibitem[Bre03]{bressan2003lemma}
Alberto Bressan.
\newblock A lemma and a conjecture on the cost of rearrangements.
\newblock {\em Rendiconti del Seminario Matematico della Universita di Padova},
  110:97--102, 2003.

\bibitem[CLS19]{Crippa2019Polynomial}
G.~Crippa, R.~Luc\`a, and C.~Schulze.
\newblock Polynomial mixing under a certain stationary {E}uler flow.
\newblock {\em Phys. D}, 394:44--55, 2019.

\bibitem[Cri17]{Crippa2017personal}
G.~Crippa.
\newblock Private communication, 2017.

\bibitem[LLN{\etalchar{+}}12]{lunasin2012optimal}
Evelyn Lunasin, Zhi Lin, Alexei Novikov, Anna Mazzucato, and Charles~R Doering.
\newblock Optimal mixing and optimal stirring for fixed energy, fixed power, or
  fixed palenstrophy flows.
\newblock {\em Journal of Mathematical Physics}, 53(11):115611, 2012.

\bibitem[LTD11]{lin2011optimal}
Zhi Lin, Jean-Luc Thiffeault, and Charles~R. Doering.
\newblock Optimal stirring strategies for passive scalar mixing.
\newblock {\em J. Fluid Mech.}, 675:465--476, 2011.

\bibitem[MMP05]{mathew2005multiscale}
George Mathew, Igor Mezi\'{c}, and Linda Petzold.
\newblock A multiscale measure for mixing.
\newblock {\em Phys. D}, 211(1-2):23--46, 2005.

\bibitem[Thi12]{thiffeault2012review}
Jean-Luc Thiffeault.
\newblock Using multiscale norms to quantify mixing and transport.
\newblock {\em Nonlinearity}, 25(2):R1, 2012.

\bibitem[YZ14]{yao2014mixing}
Yao Yao and Andrej Zlatos.
\newblock Mixing and un-mixing by incompressible flows.
\newblock {\em arXiv preprint arXiv:1407.4163}, 2014.

\end{thebibliography}
\end{document}